\documentclass[11pt]{amsart}

\newcommand{\K}{\mathcal K}
\newcommand{\M}{\mathscr M}

\newcommand{\lp}{\lambda_\M^\perp}
\newcommand{\lc}{\lambda_\M^c}
\newcommand{\uac}{\gg_u}
\newcommand{\uorth}{\perp_u}
\newcommand{\Av}[1]{\mathbf A(#1)}
\newcommand{\AM}{\mathbf A(\M)}

\newcommand{\clt}[1]{\overline{#1}^{\ *}}
\newcommand{\cl}[2]{\overline{#1}^{#2}}
\newcommand{\C}{\mathcal C}

\usepackage{amssymb}
\usepackage{enumerate}
\usepackage{amsfonts}
\usepackage{amsmath}
\usepackage{mathrsfs}
\allowdisplaybreaks[1]

\newcounter{fig}

\newcommand{\myself}{\author{Gianluca Cassese}
                     \address{Universit\`{a} Milano Bicocca and University of Lugano}
                     \email{gianluca.cassese@unimib.it}
                     \curraddr{Department of Economics, Statistics and Management, 
                                  Building U7, Room 2097, via Bicocca 
                                  degli Arcimboldi 8, 20126 Milano - Italy}}

\newtheorem{theorem}{Theorem}
\theoremstyle{plain}

\newtheorem{corollary}{Corollary}

\newtheorem{lemma}{Lemma}

\newtheorem{proposition}{Proposition}

\newcommand{\Sim}{\mathscr{S}} 
\DeclareMathOperator*{\LIM}{LIM} 
\DeclareMathOperator*{\co}{co}

\newcommand{\cls}[2]{\overline{#1}^{#2}}
\newcommand{\cco}[1]{\overline\co^{#1}}

\newcommand{\Pba}{\mathbb{P}_{ba}} 
\newcommand{\B}{\mathfrak{B}} 
\newcommand{\A}{\mathscr{A}} 

\newcommand{\R}{\mathbb{R}} 
\newcommand{\N}{\mathbb{N}}

\newcommand{\abs}[1]{\vert #1\vert} 
\newcommand{\dabs}[1]{\left\vert #1\right\vert} 

\newcommand{\net}[3]{\left\langle #1_{#2}\right\rangle_{#2\in #3} } 
\newcommand{\neta}[1]{\net{#1}{\alpha}{\mathfrak A}} 
\newcommand{\nnet}[3]{\left\langle #1\right\rangle_{#2\in #3} } 
 
\newcommand{\seq}[2]{\net{#1}{#2}{\mathbb{N}}} 
\newcommand{\sseq}[2]{\nnet{#1}{#2}{\mathbb{N}}} 
\newcommand{\seqn}[1]{\seq{#1}{n}}

\newcommand{\norm}[1]{\Vert #1\Vert} 
 
\newcommand{\cond}[3]{#1(#2\vert #3)}

\newcommand{\set}[1]{\mathbf{1}_{#1}}
\newcommand{\sset}[1]{\mathbf{1}_{\{#1\}}}

\newcommand{\emp}{\varnothing}

\newcommand{\iref}[1]{(\textit{\ref{#1}})}
\newcommand{\imply}[2]{\iref{#1}$\Rightarrow$\iref{#2}}

\begin{document}

\title[Lebesgue decomposition]{Some Implications of Lebesgue Decomposition}
\myself
\date
\today


\maketitle

\begin{abstract}
Based on a generalization of Lebesgue decomposition we obtain a characterization
of weak compactness in the space $ba(\A)$, a representation of its dual space and
some results on the structure of finitely additive measures.
\end{abstract}

\section{Introduction and Notation}
Throughout the paper $\Omega$ will be an arbitrary set, $\A$ an algebra of its subsets, 
$\lambda$ a bounded, finitely additive set function on $\A$ (i.e. $\lambda \in ba(\A)$)
and $\M\subset ba(\A)$. 

Among the well known facts of measure theory is the Lebesgue decomposition: each 
$\mu\in ba(\A)$ admits a unique way of writing 
$\lambda=\lambda_\mu^c+\lambda_\mu^\perp$ where $\lambda_\mu^c\ll\mu$ and 
$\lambda_\mu^\perp\perp\mu$. In section \ref{sec lebesgue} we prove a slight 
generalization of this classical result and use it to obtain implications on the 
properties of relatively weakly compact subsets of the space of $ba(\A)$, section 
\ref{sec compact}, and on the representation of the corresponding dual space, section 
\ref{sec riesz} and to explore some implications, section \ref{sec implications}. Eventually, 
in section \ref{sec halmos savage} we exploit Lebesgue decomposition to investigate 
some properties of dominated families of finitely additive measures.

The main, simple idea is to treat the orthogonality condition implicit 
in Lebesgue decomposition as a separating condition for subsets of $ba(\A)$ and 
to investigate its implications in the presence of some form of compactness. A
classical result associates relative weak compactness with uniform absolute
continuity. In Theorems \ref{th compact} and \ref{th compact 2}, we obtain new 
necessary and sufficient conditions for relative weak compactness of subsets of $ba(\A)$
all of which stating that a corresponding measure theoretic property has to
hold uniformly. Following from these, we then obtain, Theorem \ref{th riesz}, a 
complete characterization of the dual space of $ba(\A)$ in terms of bounded Cauchy 
nets. The Riesz representation we propose is unfortunately not as handy as that
emerging from the Riesz-Nagy Theorem for Lebesgue spaces. Nevertheless it is
helpful in some problems as those treated in Corollary \ref{cor sep}. We also
exploit it to establish a partial analogue of the Koml\'os Lemma under finite
additivity.

Likewise, the absolute continuity property implicit in Lebesgue decomposition is
exploited in section \ref{sec halmos savage} to investigate some properties of 
dominated sets of measures. We obtain the finitely additive versions of two 
classical results, due to Halmos and Savage and to Yan, respectively. Somehow 
surprisingly, these two Theorems, whose original proofs use countable additivity 
in an extensive way, carry through unchanged to finite additivity. It is also shown, 
see Theorem \ref{th hs}, that dominated families of set functions have an implicit, 
desirable property which allows to replace arbitrary families of measurable sets 
with countable subfamilies.

For the theory of finitely additive measures and integrals we mainly follow 
the notation and terminology introduced by Dunford and Schwarz \cite{bible},
although we prefer the symbol $\abs\lambda$ to denote the total variation
measure generated by $\lambda$. $\Sim(\A)$ and $\B(\A)$ designate the 
families of $\A$ simple functions, endowed with the supremum norm, and its 
closure, respectively. If $f\in L^1(\lambda)$ we denote its integral interchangeably 
as $\int fd\lambda$ or $\lambda(f)$ although, when regarded as a set function, 
we will always use the symbol $\lambda_f\in ba(\A)$. We prefer, however, 
$\lambda_B$ to $\lambda_{\set B}$ when $B\in\A$. 

We define the following families: $ba(\A,\lambda)=\{\mu\in ba(\A):\mu\ll\lambda\}$,
$ba_1(\A,\lambda)=\{\lambda_f:f\in L^1(\lambda)\}$ and $ba_\infty(\A,\lambda)=%
\{\mu\in ba(\A):\abs\mu\le c\abs\lambda\text{ for some }c>0\}$
while $\Pba(\A)$ will denote the collection of finitely additive probabilities.

The closure of $\M$ in the strong, weak and weak$^*$ topology of $ba(\A)$
is denoted by $\overline\M$, $\cl\M w$ and $\cl\M*$, respectively. We refer to $\M$ the properties 
holding for each of its elements and use the corresponding symbols accordingly. 
Thus, we write $\lambda\gg\M$ (resp. $\lambda\perp\M$) whenever 
$\lambda\gg\mu$ (resp. $\lambda\perp\mu$) for every $\mu\in\M$. $\lambda\gg\M$ 
is sometimes referred to by saying that $\M$ is dominated by $\lambda$.

\section{Lebesgue Decomposition}
\label{sec lebesgue}
Associated with $\M$ is the collection
\begin{equation}
\label{AM}
\AM=\left\{\sum_n\alpha_n\frac{\abs{\mu_n}}{1\vee\norm{\mu_n}}:\ 
\mu_n\in\M,\ \alpha_n\ge0\text{ for }n=1,2,\ldots,\sum_n\alpha_n=1\right\}
\end{equation}
as well as the set function
\begin{equation}
\label{psi}
\Psi_\M(A)=\sup_{\mu\in\M}\abs{\mu}(A)\qquad A\in\A
\end{equation}

It is at times convenient to investigate the properties of $\AM$ rather than $\M$ 
and we note to this end that $\lambda\gg\M$ (resp. $\lambda\perp\M$) is equivalent 
to $\lambda\gg\AM$ (resp. $\lambda\perp\AM$). We say that $\M$ is uniformly 
absolutely continuous (resp. uniformly orthogonal) with respect to $\lambda$, in 
symbols $\lambda\uac\M$ (resp. $\M\uorth\lambda$) whenever 
$\lim_{\abs\lambda(A)\to0}\Psi_\M(A)=0$ (resp. when for each $\varepsilon$ 
there exists $A\in\A$ such that $\Psi_\M(A)+\abs\lambda(A^c)<\varepsilon$). 
One easily verifies that either of the these uniform properties extends from $\AM$ 
to $\M$ if and only if $\M$ is norm bounded.

\begin{lemma}
\label{lemma lebesgue}
There exists a unique way of writing
\begin{equation}
\label{lebesgue}
\lambda=\lc+\lp
\end{equation}
where $\lc,\lp\in ba(\A)$ are such that (i) $m\gg\lc$ for some $m\in\Av\M$ and (ii) 
$\lp\perp\M$. If $\lambda$ is positive or countably additive then so are $\lp,\lc$.
\end{lemma}

\begin{proof}
Take an increasing net $\neta\nu$ in
\begin{equation}
\label{L(M)}
\mathcal L(\M)=\{\nu\in ba(\A):\nu\ll m\text{ for some }m\in\Av\M\}
\end{equation}
with $\nu=\lim_\alpha\nu_\alpha\in ba(\A)$. Extract a sequence $\sseq {\nu_{\alpha_n}}n$ 
such that $\norm{\nu-\nu_{\alpha_n}}=(\nu-\nu_{\alpha_n})(\Omega)<2^{-n-1}$, choose 
$m_n\in\Av\M$ such that $m_n\gg\nu_{\alpha_n}$ and define $m=\sum_n2^{-n}m_n\in\Av\M$. 
Since $m\gg\nu_{\alpha_n}$ for each $n\in\N$ so that there is $\delta_n>0$ such that 
$m(A)<\delta_n$ implies $\abs{\nu_{\alpha_n}}(A)<2^{-n-1}$ and, therefore,
$\abs\nu(A)\le\abs{\nu_{\alpha_n}}(A)+2^{-n-1}\le2^{-n}$. Thus $\mathcal L(\M)$
is a normal sublattice of $ba(\A)$ and \eqref{lebesgue} is the Riesz decomposition of $\lambda$ 
with $\lc\in\mathcal L(\M)$ and $\lp\perp\mathcal L(\M)$.
\end{proof}

Of course a different way of stating the same result is the following:

\begin{corollary}
Define $\mathcal L(\M)$ as in \eqref{L(M)}. Then, $\mathcal L(\M)=(\M^\perp)^\perp$.
\end{corollary}

Decomposition \eqref{lebesgue} gains a special interest when combined with some form 
of compactness.

\begin{lemma}
\label{lemma sep}
Let $\M\subset ba(\A)_+$ be convex and weak$^*$ compact. $\lambda\perp\M$ if
and only if $\lambda\uorth\M$.
\end{lemma}

\begin{proof}
Fix $\varepsilon>0$ and consider the set
\begin{align}
\label{Ke}
\K=\left\{f\in\Sim(\A):1\geq f\geq0,\ \abs\lambda(1-f)<\frac{\varepsilon}{4}\right\}
\end{align}
If $\lambda\perp\M$, then $\sup_{\mu\in\M}\inf_{f\in\K}\mu(f)<\varepsilon/4$.
Endow $ba(\A)$ and $\Sim(\A)$ with the weak$^*$ and the uniform topology respectively. 
Then, both $\M$ and $\K$ are convex, the former is compact and the function 
$\phi(\mu,f)=\mu(f):ba(\A)\times\Sim(\A)\to\R$ is separately linear and continuous. By a 
standard application of Sion's minimax Theorem \cite[Corollary 3.3]{sion}, there exists 
$f\in\K$ such that $\sup_{\mu\in\M}\mu(f)<\varepsilon/4$. Let $A=\{1-f<1/2\}\in\A$. 
Then Tchebiceff inequality implies $\abs\lambda(A^c)+\mu(A)<\varepsilon$ for all 
$\mu\in\M$. The converse is obvious.
\end{proof}

It is of course possible and perhaps instructive to rephrase the preceding Lemma as a 
separating condition.
\begin{corollary}
\label{cor sep}
Either one of the following mutually exclusive conditions holds: (i) $m\gg\lambda$ for some 
$m\in\Av\M$ or (ii) there exists $\eta>0$ such that for each $\M_0\subset\M$ with 
$\mathbf A(\M_0)$ weak$^*$ closed and each $k>0$ there exists $A\in\A$ for which
\begin{equation}
\label{sep}
\abs\lambda(A)>\eta>k\Psi_{\Av{\M_0}}(A)
\end{equation}
If $\A$ is a $\sigma$-algebra and $\lambda\in ca(\A)$ then \eqref{sep} rewrites 
as $\abs\lambda(A)>0=\Psi_{\Av{\M_0}}(A)$ for some $A\in\A$.
\end{corollary}

Convex, weak$^*$ compact subsets of $ba(\A)$ are often encountered in separation 
problems, where a family $\K$ of $\A$ measurable functions is given and $\M$ is the 
set $\left\{m\in\Pba(\A):\sup_{k\in\K}m(k)\le1\right\}$ of separating probabilities. In 
such special case we learn that $\lambda$ and $\M$ may be strictly separated by a 
set in $\A$.

\section{The Weak Topology}
\label{sec compact}

Decomposition \eqref{lebesgue} provides some useful insight in the study of weakly
compact subsets of $ba(\A)$. An exact characterization is the following:

\begin{theorem}
\label{th compact}
Let $\M$ be norm bounded. Then the following conditions (i)--(v) are mutually equivalent 
and imply (vi). If $\A$ is a $\sigma$ algebra and $\M\subset ca(\A)$, then (vi) implies (iii).
\begin{enumerate}[(i)]
\item\label{uac} $m\uac\M$ for some $m\in\AM$;
\item\label{comp} $\M$ is relatively weakly compact;
\item\label{umc} the set $\{\abs\mu:\mu\in\M\}$ is uniformly monotone continuous, 
i.e. if $\seqn A$ is a monotone sequence in $\A$ the limit $\lim_n\abs\mu(A_n)$ exists 
uniformly in $\M$;
\item\label{tail} for each $\M_0\subset\M$ and each sequence $\seqn A$ in $\A$ such that
\begin{equation}
\label{M sequence}
\lim_j\lim_k\abs\mu\left(\bigcup_{n=j}^{j+k}A_n\right)=0\qquad\mu\in\M_0
\end{equation}
$\mu(A_n)$ converges to $0$ uniformly with respect to $\mu\in\M_0$;
\item\label{uacp} $\M$ possesses the uniform absolute continuity property, i.e. 
$\M_0\subset\M$ and $\lambda\gg\M_0$ imply $\lambda\uac\M_0$;
\item\label{uop} $\M$ possesses the uniform orthogonality property, i.e. $\M_0\subset\M$ 
and $\lambda\perp\M_0$ imply $\lambda\uorth\M_0$.
\end{enumerate}
\end{theorem}

\begin{proof}
\imply{uac}{comp}. This is just  \cite[IV.9.12]{bible}.

\imply{comp}{umc}. Let $\seqn A$ be a decreasing 
sequence in $\A$ and define $\phi_n:ba(\A)\to\R$ by letting 
$\phi_n(\mu)=\lim_k\abs\mu(A_n\cap A_k^c)$. 
Then, $\phi_n$ is continuous and decreases to $0$ on the weak closure of $\M$ which, 
under \iref{comp}, is compact. By Dini's Theorem, convergence is uniform. 

\imply{umc}{uop}. Suppose $\lambda\perp\M$ and let $\M_1=\clt\AM$. With no 
loss of generality we can assume $\lambda\ge0$. We claim that $\lambda\perp\M_1$. 
If not then there is $m\in\M_1$ such that for some $\eta>0$ and all $A\in\A$, the 
inequality $4\eta<m(A)+\lambda(A^c)$ obtains. Fix $m_1\in\AM$ such that 
$\abs{(m-m_1)(\Omega)}<\eta/2$ and $A_1\in\A$ such that 
$m_1(A_1)+\lambda(A_1^c)<\eta$. Assume that $m_1\ldots,m_{n-1}\in\AM$ and 
$A_1,\ldots,A_{n-1}\in\A$ have been chosen such that 
\begin{equation}
\label{seq}
m_i(A_i)+\sum_{j\le i}\lambda(A_i^c)<\eta
\quad\text{and}\quad
\dabs{(m_i-m)\left(\bigcap_{j<i}A_j\right)}<\eta2^{-i}
\qquad i=1,\ldots,n-1
\end{equation}
Then pick $m_n\in\AM$ such that $\abs{(m_n-m)(\bigcap_{j<n}A_j)}<\eta2^{-n}$ 
and, by orthogonality, $A_n\in\A$ such that 
$m_n(A_n)+\lambda(A_n^c)<\eta-\sum_{k=1}^{n-1}\lambda(A_k^c)$. This proves, 
by induction that it is possible to construct two sequences $\seqn m$ in $\AM$ and 
$\seqn A$ in $\A$ that satisfy property \eqref{seq} for each $n\in\N$. It is then implicit 
that for all $n,p\in\N$
\begin{align*}
m_n\left(\bigcap_{i=1}^{n+p}A_i\right)+\lambda\left(\bigcup_{i=1}^{n+p}A_i^c\right)
&\le m_n\left(A_n\right)+\sum_{i=1}^n\lambda(A_i^c)+\sum_i\lambda(A_i^c)<2\eta
\end{align*}
and so $(m-m_n)\left(\bigcap_{i=1}^{n+p}A_i\right)>2\eta$. Observe that, under 
\iref{umc}, $\AM$ is uniformly monotone continuous and so one may fix $k$ sufficiently 
large so that $\inf_n(m-m_n)\left(\bigcap_{i=1}^kA_i\right)>\eta$, contradicting \eqref{seq}. 
Thus $\lambda\perp\M_1$ and, by Lemma \ref{lemma sep}, $\lambda\perp\Psi_{\M_1}$ 
so that $\lambda\uorth\M$. Given that property \iref{umc} extends from $\M$ to each of 
its subsets, then so does the conclusion just obtained and \iref{uop} is proved.

\imply{umc}{tail}. Let $\M_0$ and $\seqn A$ be as in \iref{tail}. Suppose that, up to the 
choice of a subsequence, there is $\varepsilon$ and a sequence $\seqn\mu$ in $\M_0$ 
such that $\abs{\mu_n}(A_n)>\varepsilon$. By \iref{umc}, for each $n$ there exists 
$k_n>n$ such that
$$
\sup_{\{\mu\in\M_0,\ p\in\N\}}\abs\mu\left(\bigcup_{i=n}^{k_n+p}A_i\right)
-\abs\mu\left(\bigcup_{i=n}^{k_n}A_i\right)<\varepsilon/2
$$
Define $\gamma\in ba(\A)$ implicitly by setting
\begin{equation}
\label{gamma}
\gamma(A)=\LIM_n\abs{\mu_n}(A_n\cap A)\qquad A\in\A
\end{equation}
where $\LIM$ denotes the Banach limit. If $B_j=\bigcup_{i=j}^{k_j}A_i$, one easily 
concludes
\begin{align*}
\gamma(B_j)&=\LIM_{n>j}\abs{\mu_n}(A_n\cap B_j)
>\LIM_{n>j}\abs{\mu_n}\left(A_n\cap \bigcup_{i=j}^{k_j+n}A_i\right)-\varepsilon/2
=\LIM_{n>j}\abs{\mu_n}(A_n)-\varepsilon/2
\ge\varepsilon/2
\end{align*}
while, under \eqref{M sequence}, $\lim_j\abs\mu(B_j)=0$. By Lemma 
\ref{lemma lebesgue}, $\gamma_{\M_0}^\perp\ne0$ and, by (\textit{\ref{uop}}),  
$\gamma_{\M_0}^\perp\uorth\M_0$ in contrast with the definition \eqref{gamma}.

\imply{tail}{uacp}. Let $\M_0\subset\M$ and $\lambda\gg\M_0$. 
For each $n\in\N$ let $A_n\in\A$ be such that $\abs\lambda(A_n)<2^{-n}$. Then 
$\sup_k\abs\lambda(\bigcup_{i=j}^kA_j)<2^{-j}$ so that, by \iref{tail}, 
$\abs\mu(A_n)$ converges to $0$ uniformly in $\M_0$.

\imply{uacp}{uac}. For each $m\in\AM$, let 
$$
\chi(m)=\sup_{\mu\in\M}\norm{\mu_m^\perp}\qquad\text{and} \qquad 
\chi(\M)=\inf_{m\in\Av\M}\chi(m)
$$ 
If $\seqn m$ is a sequence in $\AM$ such that $\chi(m_n)<\chi(\M)+2^{-n}$ and
if we define $m=\sum_n2^{-n}m_n\in\Av\M$, then from $m\gg m_n$ we conclude 
$\chi(m)=\chi(\M)$. Fix $\gamma_1=m$ and let $\mu_1\in\M$ and $A_1\in\A$ be 
such that $\gamma_1(A_1)<2^{-2}$ and $\abs{\mu_1}(A_1)\ge\chi(\M)/2$. 
Assume that $\mu_1,\ldots,\mu_{n-1}\in\M$ and $A_1,\ldots,A_{n-1}\in\A$ have been 
chosen so that, letting $\gamma_i=\frac{1}{i}(m+\abs{\mu_1}+\ldots+\abs{\mu_{i-1})}$,
\begin{equation}
\gamma_{n-1}(A_{n-1})<2^{-2(n-1)}\qquad \abs{\mu_{n-1}}(A_{n-1})\ge\chi(\M)/2
\end{equation}
Since $\gamma_n\gg m$, then $\chi(\gamma_n)=\chi(\M)$. There exists then 
$\mu_n\in\M$ such that $\norm{(\mu_n)_{\gamma_n}^\perp}>\chi(\M)/2$ and 
thus a set $A_n\in\A$ such that $\abs{\mu_n}(A_n)\ge\chi(\M)/2$ while 
$\gamma(A_n)<2^{-2n}$. It follows by induction that there are sequences $\seqn\mu$ 
and $\seqn A$ such that for all $n$
\begin{align*}
\sup_{i< n}\abs{\mu_i}(A_{n-1})<2^{-n}
\quad\text{and}\quad
\abs{\mu_n}(A_n)\ge\chi(\M)/2
\end{align*}
Let $\mu=\sum_n2^{-n}\abs{\mu_n}$. \iref{uacp} implies that the sequence 
$\seqn\mu$ is uniformly absolutely continuous with respect to $\mu$; on the other 
hand,
\begin{equation*}
\mu(A_k)=\sum_n2^{-n}\abs{\mu_n}(A_k)\le
\sum_{n=1}^k2^{-n}\abs{\mu_n}(A_k)+2^{-k}
\le2^{-(k-1)}
\end{equation*}
so that $\chi(\M)\le2\lim_k\abs{\mu_k}(A_k)=0$. But then
$\chi(m)=0$ i.e. $m\gg\M$ and, by (\textit{v}), $m\uac\M$.

\imply{uop}{umc}. Let $\A$ be a $\sigma$ algebra and $\M\subset ca(\A)$. 
Consider a decreasing sequence $\seqn B$ in $\A$ and let 
$A_n=B_n\backslash \bigcap_kB_k$. If there exists $\varepsilon>0$ and a 
sequence $\seqn\mu$ in $\M$ such that $\lim_n\abs{\mu_n}(A_n)>\varepsilon$, 
define $\gamma\in ba(\A)$ as in \eqref{gamma}. It is obvious that 
$\gamma(A_n)>\varepsilon$ so that $\gamma$ is not countably additive 
i.e. its purely fintely additive part, $\gamma^\perp$, is non zero. However, 
$\gamma^\perp\perp\M$ while, by construction, $\gamma\le\Psi_\M$, contradicting 
\iref{uop}.
\end{proof}

We also conclude 

\begin{corollary}
\label{cor compact}
Let $\M\subset ba(\A)$ be relatively weakly compact. Then, (i) $\lambda\perp\M$ if 
and only if $\lambda\perp_u\clt\AM$ and (ii) $m\in\clt\AM$ implies $m_\M^\perp=0$.
\end{corollary}

\begin{proof}
In the proof of the implication \imply{umc}{uop} of Theorem \ref{th compact} we 
showed that $\lambda\perp\M$ if and only if $\lambda\perp\clt\AM$. (\textit{i})
then follows from Corollary \ref{cor sep}; the second from (\textit{i}) and 
Lemma \ref{lemma lebesgue}.
\end{proof}

Theorem \ref{th compact} has a number of implications which help clarifying the relationship
with other well known criteria for relative weak compactness. For example, $\M$ is
relatively weakly compact if and only if $\{\abs\mu:\mu\in\M\}$ is so. Moreover, all 
disjoint sequences of sets satisfy condition \eqref{M sequence} (by boundedness) so that 
if $\M$ is relatively weakly compact then necessarily $m(A_n)$ converges to $0$ uniformly 
in $\M$ for every disjoint sequence, a property of weakly convergent sequences already 
outlined in \cite[Theorem 8.7.3]{rao}. Another immediate consequence of Theorem 
\ref{th compact} is that a subset of $ca(\A)$ is relatively weakly compact if and only if 
norm bounded and uniformly countably additive or, equivalently, uniformly absolutely 
continuous with respect to some $\lambda\in ca(\A)$, see \cite[IV.9.1 and IV.9.2]{bible}.

Another characterization of weak compactness is given in the following Theorem 
\ref{th compact 2}. A sequence $\seqn f$ in $\Sim(\A)$ is said to be uniformly bounded 
whenever $\sup_n\norm{f_n}<\infty$.

\begin{theorem}
\label{th compact 2}
In the following, conditions \iref{rwc}--\iref{un cauchy} are equivalent and imply \iref{lim}:
\begin{enumerate}[(i)]
\item\label{rwc} $\M$ is relatively weakly compact;
\item\label{un cauchy} $\M$ is bounded and possesses the uniform Cauchy property, i.e. 
if $\M_0\subset\M$ and $\seqn f$ is a uniformly bounded sequence in $\Sim(\A)$ which is 
Cauchy in $L^1(\mu)$ for all $\mu\in\M_0$, then
\begin{equation}
\label{unCa}
\lim_n\sup_{\mu\in\M_0}\sup_{p,q}\abs\mu(\abs{f_{n+p}-f_{n+q}})=0
\end{equation}
\item\label{lim} $\M$ is bounded and for each sequence $\seqn f$ as in \iref{un cauchy}
and each sequence $\seq\mu k$ in $\M$
\begin{equation}
\label{limits}
\LIM_k\lim_n\mu_k(f_n)=\lim_n\LIM_k\mu_k(f_n)
\end{equation}
\end{enumerate}
\end{theorem}

\begin{proof}
\imply{rwc}{un cauchy} If $\M$ is relatively weakly compact it is bounded and uniformly 
absolutely continuous with respect to some $m\in\AM$. If $\seqn f$ is uniformly bounded 
and Cauchy in $L^1(\mu)$ for all $\mu\in\M$, then it is Cauchy in $L^1(m)$ too. 
Moreover, given that
$$
\abs\mu(\dabs{f_{k+p}-f_{k+q}})\le 
2\sup_n\norm{f_n}\ \abs\mu^*(\dabs{f_{k+p}-f_{k+q}}\ge c)%
+c\sup_{\mu\in\M}\norm\mu
$$
\eqref{unCa} follows from uniform absolute continuity.

\imply{un cauchy}{lim} It follows from the inequality
\begin{align*}
\dabs{\LIM_k\lim_n\mu_k(f_n)-\lim_i\LIM_k\mu_k(f_i)}\le
\lim_n\sup_k\sup_{p,q}\abs{\mu_k}(\abs{f_{n+p}-f_{n+q}})
\end{align*}

\imply{un cauchy}{rwc} Choose the sequence $\seqn f$ in \iref{un cauchy} to consist 
of indicators of a decreasing sequence $\seqn A$ of $\A$ measurable sets. Then 
\eqref{unCa} implies that $\{\abs\mu:\mu\in\M\}$ is uniformly monotone continuous. 
\end{proof}

\section{The Representation of Continuous Linear Functionals on $ba(\A)$}
\label{sec riesz}

The class of sequences introduced in Theorem \ref{th compact 2} will be in this
section the basis to obtain a rather precise representation of continuous linear 
functionals on $ba(\A)$. To this end we need some additional notation. $f\in\B(\A)$ 
and $\mu\in ba(\A)$ admit the Stone space representation as 
$\tilde f\in\mathcal C(\tilde\A)$ and $\tilde\lambda\in ca(\sigma\tilde\A)$ where 
$\tilde\A$ is the algebra of all clopen sets of a compact, Hausdorff, totally disconnected 
space $\tilde\Omega$ such that $\mu(f)=\tilde\mu(\tilde f)$, \cite{bible}.

In the following we also use $\mathscr L(\A)$ for the space of continuous linear 
operators $T:ba(\A)\to ba(\A)$ and $\mathscr L_*(\A)$ for the subspace of those 
$T\in\mathscr L(\A)$ possessing the additional property
\begin{equation}
\label{T}
T(\mu_f)=T(\mu)_f
\qquad f\in L^1(\mu),\mu\in ba(\A)
\end{equation}
Remark that if $A,A_1,\ldots,A_N\in\A$ with $A_n\cap A_m=\emp$ for $n\ne m$,
then \eqref{T} implies
\begin{align*}
\sum_{n=1}^N\abs{T(\mu)(A\cap A_n)}=\sum_{n=1}^N\abs{T(\mu_{A\cap A_n})(\Omega)}
\le\norm T\sum_{n=1}^N\norm{\mu_{A\cap A_n}}
=\norm T\sum_{n=1}^N\abs{\mu}(A\cap A_n)
\le\norm T\abs\mu(A)
\end{align*}
so that $\abs{T(\mu)}\le\norm T\abs\mu$, i.e. $T(\mu)\in ba_\infty(\A,\mu)$. 
Eventually, if $T\in\mathscr L(\A)$ let $T_\lambda$ denote its restriction to 
$ba(\A,\lambda)$.
\begin{proposition}
\label{pro riesz}
$ba(\A)^*$ is isometrically isomorphic to the space $\mathscr L_*(\A)$ and the 
corresponding elements are related via the identity
\begin{equation}
\label{riesz seq}
\phi(\mu)=T(\mu)(\Omega)\qquad \mu\in ba(\A)
\end{equation}
Moreover, there is a sequence $\seqn {f^\lambda}$ in $\Sim(\A)$ uniformly bounded
by $\norm{T_\lambda}$ which is Cauchy in $L^1(\mu)$ for all $\mu\in ba(\A,\lambda)$ 
and such that
\begin{equation}
\label{riesz seq}
\limsup_n\norm{f^\lambda_n}=\norm{T_\lambda}
\quad\text{and}\quad
T(\mu)=\lim_n\mu(f^\lambda_n) \qquad \mu\in ba(\A,\lambda)
\end{equation}
If $T$ is positive then $\seqn{f^\lambda}$ can be chosen to be positive.
\end{proposition}

\begin{proof}
If $T\in\mathscr L_*(\A)$ it is obvious that the right hand side of \eqref{riesz seq} implicitly 
defines a continuous linear functional on $ba(\A)$ and 
that $\norm\phi\le\norm T$. Conversely, let $\phi\in ba(\A)^*$, fix $\mu\in ba(\A)$ 
and define the set function $T(\mu)$ on $\A$ implicitly by letting 
\begin{equation}
\label{restr}
T(\mu)(A)=\phi(\mu_A)\qquad A\in\A
\end{equation}
$T(\mu)$ is additive by the linearity of $\phi$. Moreover, if $A_1,\ldots,A_N\in\A$ are 
disjoint then
\begin{align*}
\sum_{n=1}^N\abs{T(\mu)(A\cap A_n)}
=\sum_{n=1}^N\abs{\phi(\mu_{A\cap A_n})}
\le\sum_{n=1}^N\norm\phi\norm{\mu_{A\cap A_n}}
=\sum_{n=1}^N\norm\phi\abs\mu(A\cap A_n)
=\norm\phi\abs\mu(A)
\end{align*}
so that $\abs{T(\mu)}\le\norm\phi\abs\mu$. It follows that $T(\mu)\in ba_\infty(\A,\mu)$ 
and $\norm T\le\norm\phi$. Since $(\mu_A)_B=\mu_{A\cap B}$, we conclude from 
\eqref{restr} that $T(\mu)(A\cap B)=T(\mu_A)(B)$ so that $T(\mu_A)=T(\mu)_A$ 
for all $A\in\A$. This conclusion extends by linearity to $\Sim(\A)$. If $\seqn f$ is a 
fundamental sequence for $f\in L^1(\mu)\subset L^1(T(\mu))$, then by continuity
\begin{align*}
T(\mu_f)=\lim_nT(\mu_{f_n})=\lim_nT(\mu)_{f_n}=T(\mu)_f
\end{align*}
and we conclude that $T\in\mathscr L_*(\A)$. \eqref{riesz seq} thus defines a linear 
isometry of $\mathscr L_*(\A)$ onto $ba(\A)^*$. To conclude that this is an isomorphism 
let $T_1,T_2\in\mathscr L_*(\A)$ and let $\phi_1,\phi_2$ be the associated elements of 
$ba(\A)^*$. If $T_1\ne T_2$ then $T_1(\mu)\ne T_2(\mu)$ for some $\mu\in ba(\A)$ 
and thus, by \eqref{restr}, $\phi_1(\mu_A)=T_1(\mu)(A)\ne T_2(\mu)(A)=%
\phi_2(\mu_A)$ for some $A\in\A$.

To prove \eqref{riesz seq}, denote by $\sigma:ba(\A)\to ca(\sigma\tilde\A)$ the Stone
isomorphism. Then, if $T\in\mathscr L_*(\A)$ and $\tilde T=\sigma\cdot T\sigma^{-1}$
one immediately concludes that $\tilde T:ca(\sigma\tilde\A)\to ca(\sigma\tilde\A)$ and
that 
\begin{align*}
\tilde T(\tilde\mu_{\tilde f})
&=
\lim_n\tilde T(\tilde\mu_{\tilde f_n})
=
\lim_n\sigma \left(T(\mu_{f_n})\right)
=
\lim_n\sigma\left( T(\mu)_{f_n}\right)
=
\lim_n\sigma\left( T(\mu)\right)_{\tilde f_n}
=
\lim_n\tilde T(\tilde\mu)_{\tilde f_n}
=
\tilde T(\tilde\mu)_{\tilde f}
\end{align*}
so that $\tilde T\in\mathscr L_*(\sigma\tilde\A)$. Exploiting the existence of Radon Nikodym 
derivatives we conclude that when $\mu\in ba(\A,\lambda)$,
\begin{align*}
\tilde T(\tilde\mu)=\tilde T(\tilde\lambda)_{\tilde f^\mu}
=\int\tilde f^\mu\tilde f^\lambda d\abs{\tilde\lambda}
=\int\tilde f^\lambda d\tilde\mu
\end{align*}
with $\tilde f^\mu\in L^1(\tilde\lambda)$, $\tilde f^\lambda\in L^\infty(\tilde\lambda)$
and $\norm{\tilde f^\lambda}_{L^\infty}\le\norm{T_\lambda}$. Let, as usual,
\begin{equation*}
\tilde f^\lambda_n=\sum_{i=-2^n}^{2^n}i2^{-n}\norm{T_\lambda}%
\sset{i2^{-n}\norm{T_\lambda}\le\tilde f^\lambda<(i+1)2^{-n}\norm{T_\lambda}}
\end{equation*}
The sequence $\seqn{\tilde f^\lambda}$ in $\Sim(\sigma\tilde\A)$ is increasing, converges 
uniformly to $\tilde f^\lambda$ and is positive if $T_\lambda$ is so. Replacing each 
$\sigma\tilde\A$ measurable set in the support of $\tilde f^\lambda_n$ with a corresponding 
$\tilde\A$ measurable set arbitrarily close to it in $\tilde\lambda$ measure, we obtain a 
sequence $\seqn{\hat f^\lambda}$ in $\Sim(\tilde\A)$ such that (\textit{i}) 
$\norm{\hat f^\lambda_n}\le\norm{T_\lambda}$, (\textit{ii}) $\hat f^\lambda_n$ is positive 
if $T_\lambda$ is so and 
(\textit{iii}) $\seqn{\hat f^\lambda}$ converges to $\tilde f^\lambda$ in $L^1(\tilde\mu)$
for each $\mu\in ba(\A,\lambda)$, by \cite[III.3.6]{bible}. Let 
$f_n^\lambda=\sigma^{-1}\left(\hat f^\lambda_n\right)\in\Sim(\A)$. Then,
\begin{align}
\label{conv}
T(\mu)=\sigma^{-1}\left(\tilde T(\tilde\mu)\right)
=
\lim_n\sigma^{-1}\left(\int\hat f^\lambda_nd\tilde\mu\right)
=
\lim_n\int\sigma^{-1}\left(\hat f^\lambda_n\right)d\mu
=
\lim_n\int f^\lambda_nd\mu
\end{align}
so that $\norm{T_\lambda}\le\limsup_n\norm{f^\lambda_n}$. Properties (\textit{i}) and
(\textit{ii}) carry over to the sequence $\seqn{f^\lambda}$, by the properties of the Stone
isomorphism, and therefore $\norm{T_\lambda}=\limsup_n\norm{f^\lambda_n}$. Moreover,
\begin{equation*}
\lim_n\sup_{p,q}\abs\mu\left(\dabs{f^\lambda_{n+p}-f^\lambda_{n+q}}\right)
=
\lim_n\sup_{p,q}\abs{\tilde\mu}\left(\dabs{\hat f^\lambda_{n+p}-\hat f^\lambda_{n+q}}\right)
=
0
\end{equation*}
so that the sequence is Cauchy in $L^1(\mu)$ for all $\mu\in ba(\A,\lambda)$. 
\end{proof}

Implicit in Proposition \ref{pro riesz} is a simple proof of the following, important result.

\begin{corollary}[Berti and Rigo]
\label{cor berti rigo}
The dual space of $L^1(\lambda)$ is isomorphic to $ba_\infty(\A,\lambda)$ 
and the corresponding elements are related via the identity
\begin{equation}
\label{berti rigo}
\varphi(f)=\mu(f)\qquad f\in L^1(\lambda)
\end{equation}
\end{corollary}

\begin{proof}
By the isometric isomorphism between $L^1(\lambda)$ and $ba_1(\A,\lambda)$ and
Proposition \ref{pro riesz}, each continuous linear functional $\varphi$ on $L^1(\lambda)$ 
corresponds isometrically to some $T\in\mathscr L_*(\A)$ via the identity 
$\varphi(f)=T(\lambda)(f)$. Write $\mu=T(\lambda)$. Conversely,
if $\mu\in ba_\infty(\A,\lambda)$ then it is obvious the right hand side of \eqref{berti rigo} 
defines a continuous linear functional on $L(\lambda)$.
\end{proof}

Another interesting conclusion is

\begin{corollary}
\label{cor riesz 2}
For every uniformly bounded net $\neta h$ in $\B(\A)$ there exists a 
uniformly bounded sequence $\seqn f$ in $\Sim(\A)$ which is Cauchy in $L^1(\mu)$ 
for all $\mu\in ba(\A,\lambda)$ and such that
\begin{equation}
\label{net}
\LIM_a\mu(h_a\set A)=\lim_n\mu(f_n\set A)\qquad A\in\A,\ \mu\in ba(\A,\lambda)
\end{equation}
If $\neta h$ is increasing then $\seqn f$ can be chosen to be increasing too.
\end{corollary}

\begin{proof}
The existence claim follows from Proposition \ref{pro riesz} upon noting that the left 
hand side of \eqref{net} indeed defines a continuous linear functional on $ba(\A)$.
\end{proof}

Corollary \ref{cor riesz 2} suggests that dominated families of measures admit an implicit, 
denumerable structure. This intuition will be made precise in the next section. 

An exact integral representation of the form $\phi(\mu)=\mu(f)$ for elements of $ba(\A)$
will not be possible in general, see \cite[9.2.1]{rao}. On the other hand, the representation 
\eqref{riesz seq} may seem unsatisfactory inasmuch the intervening sequence depends on 
the choice of $\lambda$. This last remark also applies to $ca(\A)$, a space for which, despite 
the characterization of weak compactness, a representation of continuous linear functionals 
is missing. The following result provides an answer.

\begin{theorem}
\label{th riesz}
A linear functional $\phi$ on $ba(\A)$ is continuous if and only if it admits the representation
\begin{equation}
\label{riesz net}
\phi(\mu)=\lim_\alpha\mu(f_\alpha)\qquad\mu\in ba(\A)
\end{equation}
where $\neta f$ is a uniformly bounded net in $\Sim(\A)$ with 
$\limsup_{\alpha\in\mathfrak A}\norm{f_\alpha}=\norm\phi$ which is Cauchy 
in $L^1(\mu)$ for all $\mu\in ba(\A)$.
\end{theorem}

\begin{proof}
It is easily seen that if the net $\neta f$ is as in the claim, then the right hand 
side of \eqref{riesz net} indeed defines a continuous linear functional on $ba(\A)$
and that $\norm\phi\le\limsup_\alpha\norm{f_\alpha}$. For the converse, 
passing to the Stone space representation and given completeness of 
$ca(\sigma\tilde\A)$, \eqref{riesz seq} becomes
\begin{equation}
\label{rep}
T(\mu)(A)=\tilde\mu(\set{\tilde A}\tilde f^\lambda)\qquad A\in\A,\ \mu\in ba(\A,\lambda)
\end{equation} 
for some $\tilde f^\lambda\in L^\infty(\tilde\lambda)$ with 
$\abs{\tilde f^\lambda}\le\norm{T_\lambda}$. Let $\mathfrak A$ be the collection of all 
finite subsets of $ba(\A)$ directed by inclusion. For each $\alpha\in\mathfrak A$ choose 
$\lambda_\alpha\in ba(\A)$ such that $\lambda_\alpha\gg\alpha$. Of course, for each 
$\mu\in ba(\A)$ there exists $\alpha\in\mathfrak A$ such that $\lambda_\alpha\gg\mu$. 
We then get the representation
\begin{equation}
\label{rep2}
T(\mu)(A)=\tilde\mu(\tilde f^{\lambda_\alpha}\set{\tilde A})
\qquad A\in\A,\ \mu\in\alpha,\ \alpha\in\mathfrak A
\end{equation}
with $\norm{\tilde f^{\lambda_\alpha}}\le\norm{T_{\lambda_\alpha}}$.
Fix $\tilde f_\alpha\in\Sim(\tilde\A)$ such that
$\norm{\tilde f_\alpha}\le\norm{\tilde f^{\lambda_\alpha}}$ and
$$
\sup_{\mu\in\alpha}
\tilde\mu\left(\dabs{\tilde f^{\lambda_\alpha}-\tilde f_\alpha}\right)\le2^{-\abs\alpha-1}
$$
and let $f_\alpha\in\Sim(\A)$ correspond to $\tilde f_\alpha$ under 
the Stone isomorphism. Then, 
\begin{align*}
\lim_\alpha\mu( f_\alpha\set A)&=\lim_{\alpha}\tilde\mu(\tilde f_\alpha\set{\tilde A})
=\lim_{\alpha}\tilde\mu(\tilde f^{\lambda_\alpha}\set{\tilde A})
=T(\mu\set A)
\qquad A\in\A,\ \mu\in ba(\A)
\end{align*}
which, together with \eqref{riesz net}, proves the existence of the representation 
\eqref{riesz seq} and of the inequality 
$\limsup_\alpha\norm{f_\alpha}\le\lim_\alpha\norm{T_{\lambda_\alpha}}%
\le\norm T=\norm\phi$. Moreover, if $\alpha_1,\alpha_2,\alpha\in\mathfrak A$ and 
$\mu\in\alpha\subset\alpha_1,\alpha_2$, then
\begin{align*}
\mu(\dabs{f_{\alpha_1}-f_{\alpha_2}})
&=\mu\left(h(\alpha_1,\alpha_2)(f_{\alpha_1}-f_{\alpha_2})\right)\\
&\le2^{-\abs\alpha}+\tilde\mu\left(\tilde h(\alpha_1,\alpha_2)(\tilde f^{\lambda_{\alpha_1}}
-\tilde f^{\lambda_{\alpha_2}})\right)\\
&=2^{-\abs\alpha}
\end{align*}
the third line following from \eqref{rep2} and the inclusion $h(\alpha_1,\alpha_2)\in\Sim(\A)$. 
But then $\neta f$ is indeed a Cauchy net in $L^1(\mu)$ for all $\mu\in ba(\A)$. 
\end{proof}

The space of uniformly bounded nets in $\Sim(\A)$ is a linear space if, for
$\tilde f=\neta f$ and $\tilde g=\net g\delta{\mathfrak D}$ two such 
nets, we endow $\mathfrak A\times\mathfrak D$ with the product order obtained by letting 
$(\alpha_1,\delta_1)\ge(\alpha_2,\delta_2)$ whenever $\alpha_1\ge\alpha_2$ and 
$\delta_1\ge\delta_2$ and write $\tilde f+\tilde g$ as 
$\nnet{f_\alpha+g_\delta}{(\alpha,\delta)}{\mathfrak A\times\mathfrak D}$. Theorem 
\ref{th riesz} suggests the definition of a seminorm on such space by letting
\begin{equation}
\label{norm}
\norm{F}=\limsup_\alpha\norm{f_\alpha}
\qquad\text{whenever}\qquad 
F=\neta f
\end{equation}
and denote by $\mathfrak C(\A)$ the linear space of equivalence classes of 
uniformly bounded nets in $\Sim(\A)$ which are Cauchy in $L^1(\mu)$ for all 
$\mu\in ba(\A)$.

\begin{theorem}
The identity \eqref{riesz net} defines an isometric isomorphism between $ba(\A)^*$ 
and $\mathfrak C(\A)$.
\end{theorem}

\begin{proof}
The right hand side of \eqref{riesz net} is invariant upon replacing the net $F=\neta f$ 
with $G=\net g\delta{\mathfrak D}$ whenever $\norm{F-G}=0$.
\end{proof}

\section{Some Implications.}
\label{sec implications}

The characterization so obtained is admittedly not an easy one, due to the intrinsic
difficulty of identifying explicitly the net associated to each continuous functional.
It has, this notwithstanding, a number of interesting implications. We illustrate
some with no claim of completeness.

\begin{corollary}
\label{cor sep}
Let $\M$ and $\mathscr N$ be convex, weakly compact subsets of $ba(\A)$. Then, 
\begin{enumerate}[(i)]
\item
$\M\cap\mathscr N=\emp$ if and only if there exists 
$f\in\Sim(\A)$ such that $\inf_{\nu\in\mathscr N}\nu(f)>\sup_{\mu\in\M}\mu(f)$;
\item
there exists $\K\subset\Sim(\A)$ and a subset $\M_0$ of extreme points of $\M$ 
such that
\begin{equation}
\label{support}
\M=\left\{m\in ba(\A):m(k)\le\max_{\mu\in\M_0}\mu(k)\text{ for all }k\in\K\right\}
\end{equation}
\end{enumerate}
\end{corollary}

\begin{proof}
(\textit{i}). 
The weak topology is linear. There is then a linear functional $\phi$ on $ba(\A)$ 
and constants $a_1>b_1$ such that 
$\inf_{\nu\in\mathscr N}\phi(\nu)>a>b>\sup_{\mu\in\M}\phi(\mu)$.
By compactness, $\M$ and $\mathscr N$ are dominated so that, by Proposition 
\ref{pro riesz}, $\phi$ is associated with a uniformly bounded, Cauchy sequence 
$\seqn f$ in $\Sim(\A)$, as in \eqref{riesz seq}. We also know from Theorem 
\ref{th compact 2} that for all $\varepsilon$ there exists $n$ sufficiently
large so that $\inf_{\nu\in\mathscr N}\nu(f_n)>a_1-\varepsilon$ and 
$b_1+\varepsilon>\sup_{\mu\in\M}\mu(f_n)$. Choosing 
$\varepsilon<(a-b)/2$ we get 
$\inf_{\nu\in\mathscr N}\nu(f_n)>\frac{a+b}{2}\sup_{\mu\in\M}\mu(f_n)$.

(\textit{ii}). 
For each $m\notin\M$ there is then $k_m\in\Sim(\A)$ such that 
$\sup_{\mu\in\M}\mu(k_m)<m(k_m)$. Let $\K=\{k_m:m\notin\M\}$.
For each $k\in\K$ choose one extreme point $\mu_k\in\M$ in the corresponding 
supporting set of $\M$ and let $\M_0=\{\mu_k:k\in\K\}$. By construction, each
$k\in\K$, when considered as a function on $\M$, attains its maximum on $\M_0$, 
so that the right hand side of \eqref{support} contains $\M$. For each $m\notin\M$ 
there is $k\in\K$ such that $m(k)>\sup_{\mu\in\M}\mu(k)$ so that the right hand 
side is included in $\M$.
\end{proof}

It is well known that, combining the Theorems of Eberlein Smulian and of Mazur, 
and taking convex combinations one may transform a weakly convergent sequence
in a Banach space into a norm convergent one. The following result establishes
a weak form of this fundamental result which holds even in the absence of weak
convergence. The proof exploits some of the ideas introduced by Koml\'{o}s
\cite{komlos}.

\begin{theorem}
\label{th komlos}
Let $\seqn\mu$ be a norm bounded sequence in $ba(\A)_+$ and define
\begin{equation}
\Gamma(n)=\co(\mu_n,\mu_{n+1},\ldots)
\quad\text{and}\quad
\lambda=\sum_n2^{-n}\mu_n
\end{equation}
There exists $\xi\in ba(\lambda)_+$ and a sequence $\seqn m$ with 
$m_n\in\Gamma(n)$ for each $n\in\N$ such that 
\begin{equation}
\label{fatou}
\lim_n\norm{(\xi\wedge k\lambda)-(m_n\wedge k\lambda)}=0
\quad\text{and}\quad
\xi(A)\le\liminf_nm_n(A)
\qquad
A\in\A
\end{equation}
\end{theorem}

\begin{proof}
Define the families
\begin{align*}
\mathscr C(n)
=
\{\nu\in ba(\A)_+:\nu\le m\text{ for some }m\in\Gamma(n)\}
\quad\text{and}\quad
\mathscr C=\bigcap_n\cls{\mathscr C(n)}{ }
\end{align*}
and the set functions
\begin{align}
\label{nu}
\nu_k=\LIM_n(\mu_n\wedge k\lambda)
\qquad k\in\N
\end{align}
We notice that the sequence $\tilde\nu=\seq{\nu}{k}$ so obtained satisfies
the following properties for all $k\in\N$: (\textit{a}) $\nu_{k-1}\le\nu_k\le k\lambda$ 
and (\textit{b}) $\nu_k\in\mathscr C$ as, by Theorem \ref{th exchange}, 
\begin{align*}
\nu_k
\in
\bigcap_n\cco{}(\mu_n\wedge k\lambda,\mu_{n+1}\wedge k\lambda,\ldots)
\end{align*}
The family $\Xi$ of sequences possessing properties (\textit{a}) and (\textit{b})
may be partially ordered upon letting $\tilde\nu\ge\tilde\nu'$ whenever
$\nu_k\ge\nu'_k$ for all $k\in\N$. If $\{\tilde\nu^a:a\in\mathfrak A\}$
is a chain in $\Xi$ we may let $\nu_k=\lim_a\nu^a_k$ and
$\tilde\nu=\seq{\nu}{k}$. It is easily seen that $\tilde\nu$ is increasing
and that $\nu_k\le k\lambda$. Moreover, since $\mathscr C$ is closed 
and norm bounded, $\nu_k^a$ converges to $\nu_k$ in norm, so that 
$\nu_k\in\mathscr C$. Let $\tilde\xi$ be a maximal element in $\Xi$ and define 
the set function $\xi$ implicitly by letting $\xi(A)=\lim_k\xi_k(A)$ for each 
$A\in\A$. Observe that $\xi(\Omega)\le\sup_n\norm{\mu_n}<\infty$ so that
$\lim_k\norm{\xi-\xi_k}=0$ and $\xi\in\mathscr C$. There exist then 
sequences $\seqn m$, with $m_n\in\Gamma(n)$, and $\seqn\delta$ in $ba(\A)_+$ 
such that $\kappa_n=m_n-\delta_n\in\mathscr C(n)$ and that 
$\lim_n\norm{\kappa_n-\xi}=0$. But then
\begin{align}
\label{dom}
\xi_k
\le
\xi\wedge k\lambda
=
\lim_n(\kappa_n\wedge k\lambda)
\le
\LIM_n(m_n\wedge k\lambda)
\equiv
\xi'_k
\qquad
k\in\N
\end{align}
where the intervening limits refer to setwise convergence. Observe that 
$\tilde\xi'=\seq{\xi'}k\in\Xi$ and $\tilde\xi'\ge\tilde\xi$ which is contradictory 
unless $\xi_k=\xi\wedge k\lambda=\LIM_n(m_n\wedge k\lambda)$. This
clearly implies $\xi(A)\le\liminf_nm_n(A)$ for each $A\in\A$. Moreover, 
\eqref{dom} remains true if we replace the sequence $\seqn m$ with any of its
subsequences. This implies that $(m_n\wedge k\lambda)(A)$ converges to $\xi_k(A)$
for each $A\in\A$ so that $\LIM_n(m_n\wedge k\lambda)=\lim_n(m_n\wedge k\lambda)$.
From the inequality 
$
\delta_n
\le
\abs{\kappa_n-\xi}+m_n-\xi_k
$
we deduce
\begin{align*}
\delta_n\wedge k\lambda
\le
\abs{\kappa_n-\xi}+((m_n-\xi_k)\wedge k\lambda)
\le
\abs{\kappa_n-\xi}+(m_n\wedge2k\lambda)-\xi_k
\end{align*}
and therefore
\begin{align*}
\limsup_n\norm{\delta_n\wedge k\lambda}
&\le
\lim_k\limsup_n(\delta_n\wedge k\lambda)(\Omega)\\
&\le
\lim_k\{\limsup_n(m_n\wedge2k\lambda)(\Omega)-\xi_k(\Omega)\}\\
&\le
\lim_k(\xi_{2k}-\xi_k)(\Omega)\\
&=0
\end{align*}
But then we conclude
\begin{align*}
\limsup_n\norm{(m_n\wedge k\lambda)-(\xi\wedge k\lambda)}
=
\limsup_n\norm{(m_n\wedge k\lambda)-(\kappa_n\wedge k\lambda)}
\le
\limsup_n\norm{\delta_n\wedge k\lambda}
=
0
\end{align*}
Property (\textit{ii}) in the claim is a clear consequence of \eqref{dom} and of the
fact that $\xi\ll\lambda$.
\end{proof}
It is clear from the proof that the condition $\liminf_n\norm{\lambda_n}<\infty$ may 
be replaced with the inequality
$\lim_k\lim_n\sup_{\mu\in\Gamma(n,k)}\norm\mu<\infty$,
which is more general but less perspicuous.

One should also remark that if the sequence $\seqn\lambda$ in Theorem 
\ref{th komlos} is weakly convergent, then, by the uniform absolute continuity property, 
$\mu_n\wedge k\lambda$ converges (in norm) to $\mu_n$ uniformly in $n\in\N$ and 
thus the sequence $\seqn\mu$ converges strongly to $\xi$. Theorem \ref{th komlos}
is then indeed a generalization of more classical results. Some implications of Theorem
\ref{th komlos} are developed in \cite{BK}.

\section{The Halmos-Savage Theorem and its Implications}
\label{sec halmos savage}

The  results of the preceding section mainly develop the orthogonality implications of Lemma 
\ref{lemma lebesgue}. We may as well deduce interesting conclusions concerning 
absolute continuity, among which the following finitely additive version of the Lemma 
of Halmos and Savage \cite[Lemma 7, p. 232]{halmos savage}. 

\begin{theorem}[Halmos and Savage]
\label{th halmos savage}
$\M\subset ba(\A,\lambda)$ if and only if $\M\subset ba(\A,m)$ for some $m\in\Av\M$.
\end{theorem}

\begin{proof}
$\lambda$ dominates $\M$ if and only if $\lc$ does. The claim follows from 
Lemma \ref{lemma lebesgue}.
\end{proof}

As is well known, Halmos and Savage provided applications of this result to 
the theory of sufficient statistics. Another possible development is the following 
finitely additive version of a well known Theorem of Yan \cite[Theorem 2, p. 220]{yan}:

\begin{corollary}[Yan]
\label{cor yan}
Let $\K\subset L^1(\lambda)$ be convex with $0\in\K$, $\C=\K-\Sim(\A)_+$ 
and denote by $\overline{\C}$ the closure of $\C$ in $L^1(\lambda)$.
The following are equivalent: 
\begin{enumerate}[(i)]
\item\label{eta f} for each $f\in L^1(\lambda)_+$ with $\abs\lambda(f)>0$ 
there exists $\eta>0$ such that $\eta f\notin\overline{\C}$;
\item\label{d A} for each $A\in\A$ with $\abs\lambda(A)>0$ there exists 
$d>0$ such that $d\set A\notin\overline{\C}$;
\item\label{m} there exists $m\in\Pba(\A)$ such that (a) $\K\subset L^1(m)$ 
and $\sup_{k\in\K}m(k)<\infty$, (b) $m\in ba_\infty(\A,\lambda)$ and (c) $m(A)=0$ 
if and only if $\abs\lambda(A)=0$.
\end{enumerate}
\end{corollary}

\begin{proof}
The implication \imply{eta f}{d A} is obvious. If $A$ and $d$ are as in \iref{d A} 
there exists a continuous linear functional $\phi^A$ on $L^1(\lambda)$ separating 
$\{d\set A\}$ and $\overline{\C}$ and $\phi^A$ admits the representation 
$\phi^A(f)=\mu^A(f)$ for some $\mu^A\in ba_\infty(\A,\lambda)$ such that 
$\mu^A\le c^A\abs\lambda$, Corollary \ref{cor berti rigo}. Thus 
$\sup_{h\in\C}\mu^A(f)\le a<b<d\mu^A(A)$. The inclusion $0\in\C$ implies 
$a\ge0$ so that $\mu^A(A)>0$; moreover, $\mu^A\ge0$ as $-\Sim(\A)_+\subset\C$. 
By normalization we can assume $\norm{\mu^A}\vee c^A\vee a\le1$. The collection 
$\M=\{\mu^A:A\in\A,\ \abs\lambda(A)>0\}$ so obtained is dominated by 
$\lambda$ and therefore by some $m\in\Av\M$, by Theorem \ref{th halmos savage}. 
Thus $m\le\abs\lambda$, $\norm m\le1$ and $\sup_{h\in\C}m(h)\le1$. If $A\in\A$ 
and $\abs\lambda(A)>0$ then $m\gg\mu^A$ implies $m(A)>0$. By normalization 
we can take $m\in\Pba(\A)$. Let $m$ be as in \iref{m} so that 
$L^1(\lambda)\subset L^1(m)$. If $f\in L^1(\lambda)_+$ and $\abs\lambda(f)>0$ 
then $f\wedge n$ converges to $f$ in $L^1(\lambda)$ \cite[III.3.6]{bible} so that 
we can assume that $f$ is bounded. Then, by \cite[4.5.7 and 4.5.8]{rao} there exists 
an increasing sequence $\seqn f$ in $\Sim(\A)$ with $0\le f_n\le f$ such that $f_n$ 
converges to $f$ in $L^1(\lambda)$ and therefore in $L^1(m)$ too. For $n$ large 
enough, then, $\abs\lambda(f_n)>0$ and, $f_n$ being positive and simple, 
$m(f_n)>0$. But then $m(f)=\lim_nm(f_n)>0$ so that $\eta f$ cannot be an element 
of $\overline\C$ for all $\eta>0$ as $\sup_{h\in\overline\C}m(h)<\infty$.
\end{proof}

An application of Corollary \ref{cor yan} is obtained in \cite{BK}.

One may also draw from Theorem \ref{th halmos savage} some implications on the 
structure of a finitely additive set function. 

\begin{theorem}
\label{th hs}
Let $\M\subset ba(\A,\lambda)$ and let $\mathscr H_0\subset\A$ generate the ring
$\mathscr H$. There exist $H_1,H_2,\ldots\in\mathscr H_0$ such that, letting 
$G_n=H_n\backslash\bigcup_{k<n}H_k$ and $G=\bigcap_nH_n^c$, the following holds:
\begin{align}
\label{disintegrate H}
\abs\mu^*(H\cap G)=0\quad\text{and}\quad
\mu(A\cap H)=\sum_n\mu\left(A\cap H\cap G_n\right)
\qquad\mu\in\M,\ A\in\A,\ H\in\mathscr H
\end{align}
Moreover: (i) if $\mu\in\M$ is $\mathscr H_0$-inner regular then 
\begin{equation}
\label{disintegrate A}
\mu(A)=\sum_n\mu\left(A\cap G_n\right)
\qquad A\in\A
\end{equation}
(ii) if $\mathscr H_0$ is closed with respect to countable unions then
\begin{equation}
\label{disintegrate A0}
\mu(A)=\mu(A\cap G)+\sum_n\mu(A\cap G_n)\qquad \mu\in\M,\ A\in\A
\end{equation}
\end{theorem}

\begin{proof}
With no loss of generality, let $\lambda\ge0$ and write $\M=\{\lambda_H:H\in\mathscr H_0\}$. 
By Theorem \ref{th halmos savage}, choose $m_0=\sum_n\alpha_n\lambda_{H_n}\in\Av\M$ 
to be such that $m_0\gg\M$. Let $G$ and $G_n$ be as in the statement and define
$m=\sum_n\lambda_{G_n}$. Observe that 
$m\ge m_0$ and that, by construction, $\lim_km(\bigcap_{n<k}H_n^c)=0$. But then, 
for each $H\in\mathscr H_0$ we conclude $\lim_k\lambda_H(\bigcap_{n<k}H_n^c)%
=\lim_k\lambda(H\cap\bigcap_{n<k}H_n^c)=0$ and, by absolute continuity,
$\abs\mu^*(H\cap G)\le\lim_k\abs\mu(H\cap\bigcap_{n<k}H_n^c)=0$ for all 
$\mu\in\M$. Consequently, if $A\in\A$ and $H\in\mathscr H_0$
\begin{align*}
\mu(A\cap H)
&=\mu\left(A\cap H\cap\left(\bigcup_{n< k}G_n\cup\bigcap_{n< k}G^c_n\right)\right)\\
&=\lim_k\mu\left(A\cap H\cap\bigcup_{n< k}G_n\right)\\
&=\sum_n\mu(A\cap H\cap G_n)
\end{align*}
The set function $\sum_n\mu_{G_n}$ agrees with $\mu$ on the ring $\mathscr R$ 
consisting of all finite, disjoint unions of sets of the form $A\cap H$ with $A\in\A$ and 
$H\in\mathscr H_0$. Another ring is the collection $\mathscr J=%
\{H\in\mathscr H:A\cap H\in\mathscr R\ \text{ for all }A\in\A\}$ which therefore coincides 
with $\mathscr H$. Thus, $\{H\cap A:H\in\mathscr H,\ A\in\A\}\subset\mathscr R$
which proves \eqref{disintegrate H}.

If $\mu\in\M$ is $\mathscr H_0$-inner regular, then,
\begin{align*}
\mu^+(A)=\sup_{\{H\in\mathscr H_0:H\subset A\}}\mu(H)
=\sup_{\{H\in\mathscr H_0:H\subset A\}}\ \sum_n\mu(H\cap G_n)
\le\sum_n\mu^+(A\cap G_n)
\le\mu^+(A)
\end{align*}
the last inequality following from additivity. Exchanging $\mu$ with $-\mu$ proves 
\eqref{disintegrate A}. Eventually, if $\mathscr H_0$ is closed with respect to countable 
unions, then $\bigcup_{n>k}G_n\in\mathscr H$ and, by \eqref{disintegrate H}, 
$\mu\left(\bigcup_{n>k}G_n\right)=\sum_{n>k}\mu(G_n)$ from which 
\eqref{disintegrate A0} readily follows. 
\end{proof}

The following Corollary \ref{cor borel} illustrates a special case.

\begin{corollary}
\label{cor borel}
Let $\Omega$ be a separable metric space, $\A$ its Borel $\sigma$-algebra and 
$\M\subset ca(\A,\lambda)$. If $\pi$ is a partition of $\Omega$ into open sets then
there exist $H_1,H_2,\ldots\in\pi$ such that
\begin{equation}
\label{disintegrate borel}
\mu(A)=\sum_n\mu(A\cap H_n)\qquad A\in\A,\ \mu\in\M
\end{equation}
\end{corollary}
\begin{proof}
Under the current assumptions, for each increasing net $\neta O$ of open sets we 
have $\lambda(\bigcup_\alpha O_\alpha)=\lim_\alpha\lambda(O_\alpha)$, 
\cite[Proposition 7.2.2]{bogachev}. Let $\mathscr H_0=\pi$ extract $H_1,H_2,\ldots\in\pi$ 
as in Theorem \ref{th hs} and observe that, $\pi$ being a partition, $G_n=H_n$ for 
$n=1,2,\ldots$; moreover $G=\bigcup_{H\in\pi,H\subset G}H$ and so $\lambda(G)=0$. 
We conclude that \eqref{disintegrate borel} holds. 
\end{proof}

To motivate further our interest in the preceding conclusions, assume that $\pi$
is an $\A$ partition and that $\lambda$ is $\pi$-inner regular. Then for each $H\in\pi$ 
and $A\in\A$ one may define $\cond\sigma A H=\lambda(A\cap H_n)/\lambda(H_n)$ 
if $H=H_n$ and $\lambda(H_n)\ne0$ or $\cond\sigma A H=m_H(A)$ for any 
$m_H\in\Pba(\A)$ with $m_H(H)=1$. Write 
$\cond\sigma A \pi=\sum_{H\in\pi}\cond\sigma A H \set H$. Then,
\begin{equation}
\lambda(A)=\int\cond\sigma A\pi d\lambda\qquad A\in\A
\end{equation}
This follows from $\int\cond\sigma A\pi d\lambda=%
\sum_n\cond\sigma A {G_n}\lambda(G_n)+\int_{G}\cond\sigma A\pi d\lambda
=\sum_n\lambda(A\cap G_n)
=\lambda(A)
$.
In the terminology introduced by Dubins \cite{dubins}, $\lambda$ is then strategic
along any partition relatively to which it is inner regular.


\begin{thebibliography}{9}
\bibitem{berti rigo} P. Berti, P. Rigo: \textit{Integral Representation of Linear Functionals
on Spaces of Unbounded Functions }, Proc. Amer. Math. Soc. \textbf{126}, 3251-3258.

\bibitem{rao} K. P. S. Bhaskara Rao, M. Bhaskara Rao: \textit{Theory of Charges}, 
Academic Press, London , 1983.

\bibitem{bogachev} V. I. Bogachev: \textit{Measure Theory}, Springer-Verlag, Berlin-
Heidelberg, 2007.

\bibitem{BK} G. Cassese: \textit{Convergence in Measure under Finite Additivity}, 
 arXiv:1203.6768 (2012).

\bibitem{dubins} L. E. Dubins: \textit{Finitely Additive Conditional Probabilities, Conglomerability
and Disintegrations}, Ann, Probab. \textbf{3} 1975, 89-99.

\bibitem{bible} N. Dunford, J. Schwartz: \textit{Linear Operators. General Theory}, Wiley, 
New York, 1988.

\bibitem{halmos savage} P. R. Halmos, L. J. Savage: \textit{Application of the Radon-%
Nikodym Theorem to the Theory of Sufficient Statistics }, Ann. Math. Stat. \textbf{20} 
(1949), 225-241.

\bibitem{komlos} J. Koml\'{o}s: \textit{A Generalization of a Problem of Steinhaus},
Acta Math. Hung. \textbf{18} (1967), 217-229.

\bibitem{sion} M. Sion: \textit{On General Minimax Theorems}, Pacific J. Math. \textbf 8 
(1958), 171-175.

\bibitem{yan}  J. A. Yan: \textit{Caract\'{e}risation d'une Classe d'Ensembles Convexes de 
$L^{1}$ ou $H^{1}$}, S\'{e}minaire de Probabilit\'{e} XIV, Lecture Notes in Math 
\textbf{784} (1980), 220-222.
\end{thebibliography}
\end{document}